\documentclass[12pt]{amsart}
\usepackage{amssymb}
\usepackage{graphicx}
\newcommand{\N}{{\mathbb N}}
\newcommand{\C}{{\mathbb C}}

\newcommand{\D}{{\mathbb D}}
\newcommand{\Ha}{{\mathbb H}}
\newcommand{\capa}{\rm cap\,}
\newcommand{\are}{\rm area\,}

\newcommand{\wand}{wandering domain}
\newcommand{\bd}{Baker domain}
\newcommand{\tef}{transcendental entire function}

\newcommand{\sconn}{simply connected}

\newcommand\eps{\varepsilon}

\theoremstyle{plain}
\newtheorem{theorem}{Theorem}[section]
\newtheorem{corollary}{Corollary}[section]

\theoremstyle{definition}

\theoremstyle{remark}

\theoremstyle{problem}

\theoremstyle{example}

\begin{document}

\title[Boundaries of univalent Baker domains]{Boundaries of univalent Baker domains}

\author{P. J. Rippon}
\address{Department of Mathematics and Statistics\\
The Open University \\
Walton Hall\\
Milton Keynes MK7 6AA\\
UK}
\email{p.j.rippon@open.ac.uk}

\author{G. M. Stallard}
\address{Department of Mathematics and Statistics\\
The Open University \\
Walton Hall\\
Milton Keynes MK7 6AA\\
UK}
\email{g.m.stallard@open.ac.uk}

%\thanks{The second author <... thanks>}

%End Two authors

%\keywords{asymptotic tracts}
\subjclass{30D05, 37F10.\newline\hspace*{.32cm} Both authors are
supported by EPSRC grant EP/K031163/1.}
%\date{1999}

%End topmatter

\begin{abstract}
Let~$f$ be a {\tef} and let~$U$ be a univalent Baker domain of~$f$. We prove a new result about the boundary behaviour of conformal maps and use this to show that the non-escaping boundary points of~$U$ form a set of harmonic measure zero with respect to~$U$. This leads to a new sufficient condition for the escaping set of~$f$ to be connected, and also a new general result on Eremenko's conjecture.
\end{abstract}

\maketitle

\section{Introduction}
\setcounter{equation}{0}
Let $f$ be a {\tef} and denote by $f^{n},\,n=0,1,2,\ldots\,$, the
$n$th iterate of~$f$. The {\it Fatou set} $F(f)$ is defined to be
the set of points $z \in \C$ such that $(f^{n})_{n \in \N}$ forms
a normal family in some neighborhood of $z$.  The components of
$F(f)$ are called {\it Fatou components}. The complement of
$F(f)$ is called the {\it Julia set} $J(f)$. An introduction to
the properties of these sets can be found in~\cite{wB93}.

The set $F(f)$ is completely invariant, so for any component $U$ of $F(f)$ there exists, for each $n=0,1,2,\ldots\,$, a component
of $F(f)$, which we call $U_n$, such that $f^n(U) \subset U_n$. If, for some $p\ge 1$, we have $U_p =U_0= U$, then we say that
$U$ is a {\it periodic component} of {\it period} $p$, assuming $p$ to be minimal. There are then four possible types of periodic
components; see \cite[Theorem~6]{wB93} for a classification.

The {\it escaping set}
\[
 I(f) = \{z: f^n(z) \to \infty \mbox{ as } n \to \infty \}
\]
was first studied in detail by Eremenko~\cite{aE89} who showed that $I(f)\ne\emptyset$ and indeed that $I(f)\cap J(f) \ne\emptyset$, and also made what is known as `Eremenko's conjecture' which states that all the components of $I(f)$ are unbounded.

Any Fatou component that meets $I(f)$ must lie in $I(f)$ by normality. A periodic Fatou component in $I(f)$ is called a {\it Baker domain}; see \cite{pR08}, for example, for the properties of this type of Fatou component, in particular that a Baker domain of a {\tef} must be unbounded and simply connected.  For the function $f(z)=z+1+e^{-z}$, studied by Fatou in \cite{pF26}, the set $F(f)$ is a completely invariant Baker
domain, whose boundary is $J(f)$. In this case the Baker domain has many boundary points in $I(f)$ and many that are {\it not} in~$I(f)$.

It is natural to ask whether every {\bd} of a {\tef} must have at least one boundary point in
$I(f)$. In \cite[Remark following the proof of Theorem~1.1]{RS11} we showed that if $U$ is an invariant Baker domain in which there is an orbit
$z_n=f^n(z_0)$, $n\in\N$, such that, for some $k>1$,
\begin{equation}\label{kcond}
|z_{n+1}|\ge k|z_n|,\quad\text{for }n\in\N,
\end{equation}
then $\partial U\cap I(f)^c$ has harmonic measure zero relative to~$U$. Since a Baker domain of~$f$ of period~$p$ is an invariant Baker domain of $f^p$ and~$f$ maps any boundary point of a Fatou component to a boundary point of a Fatou component, the analogous result holds for periodic Baker domains.

In \cite{BF01} Bara\'nski and Fagella studied {\tef}s with univalent Baker domains. A Baker domain~$U$ of~$f$ of period~$p$ is said to be {\it univalent} if $f^p$ is univalent in~$U$. In this paper we prove the following result about such Baker domains.

\begin{theorem}\label{thm1.1}
Let $f$ be a {\tef} and let~$U$ be a univalent {\bd} of~$f$. Then $\partial U\cap I(f)^c$ has harmonic measure zero relative to~$U$.

More precisely, if $\phi$ is a conformal map of the open unit disc~$\D$ onto~$U$, then for all $\zeta\in\partial \D$ apart from a set of capacity zero the angular limit $\phi(\zeta)$ exists and lies in $\partial U\cap I(f)$.
\end{theorem}
The first statement of Theorem~\ref{thm1.1} follows from the second because the set of angular limits of the conformal map~$\phi$ forms the set of accessible boundary points of~$U$, which has full harmonic measure in $U$, and any set of boundary points of $\partial \D$ of linear measure zero (in particular, of capacity zero) gives rise to a set of accessible boundary points of~$U$ of harmonic measure zero; see \cite[page~206]{GM05}.

The paper \cite{BF01} gives several examples of univalent {\bd}s: some with connected boundaries and some with disconnected boundaries; some for which \eqref{kcond} holds and some for which it does not hold. Baker and Dom\'inguez \cite[Corollary~1.3]{BD99} showed that the boundary of any non-univalent Baker domain is disconnected, and indeed has uncountably many components, so we have the following immediate corollary of Theorem~\ref{thm1.1}.

\begin{corollary}\label{cor1.1}
Let~$f$ be a {\tef} and let~$U$ be a {\bd} of~$f$ whose boundary is connected. Then $\partial U\cap I(f)^c$ has harmonic measure zero relative to~$U$.
\end{corollary}

It remains an intriguing open question whether $\partial U\cap I(f)\ne \emptyset$ whenever $U$ is a Baker domain. Note that in \cite{RS11} we showed that if~$U$ is any wandering domain in $I(f)$, then almost all points of $\partial U$, with respect to harmonic measure in $U$, are escaping.

We prove Theorem~\ref{thm1.1} using a new result on the boundary behaviour of conformal maps, which we state and prove in Section~3.

Using Theorem~\ref{thm1.1} together with \cite[Theorem~5.1]{RS11}, we can give a new sufficient condition for $I(f)$ to be connected, and so satisfy Eremenko's conjecture in a particularly strong way.

\begin{theorem}\label{thm1.2}
Let~$f$ be a {\tef} and let~$E$ be a set such that $E\subset I(f)$ and $J(f)\subset\overline{E}$. Either $I(f)$ is connected or it has infinitely many components that meet~$E$; in particular, if $E$ is connected, then $I(f)$ is connected.
\end{theorem}
Several subsets of $I(f)$ have been studied, involving different rates of escape, including: the {\it fast escaping set} $A(f)$ (see \cite{BH99} and \cite{RS09}), the {\it slow escaping set} $L(f)$  and {\it moderately slow escaping set} $M(f)$ (see \cite{RS09a}),  the {\it quite fast escaping set} $Q(f)$ (see \cite{RS12c}), $Z(f)$ (see \cite{RS00}) and $I'(f)$ (see \cite{wB95a}). Each of these sets contains at least three points and is backwards invariant, so its closure contains $J(f)$ by Montel's theorem. Thus we obtain the following corollary of Theorem~\ref{thm1.2}.
\begin{corollary}
Let $f$ be a {\tef}. If one of the sets $A(f)$,~$L(f)$,~$M(f)$,~$Q(f)$, $Z(f)$ or~$I'(f)$ is connected, then I(f) is connected.
\end{corollary}
The fast escaping set $A(f)$ has the property that all its components are unbounded \cite[Theorem~1.1]{RS09}. Therefore, if we apply Theorem~\ref{thm1.2} in the case when the set $E$ is $A(f)$, then we obtain the following result, which seems to be the strongest general result so far on Eremenko's conjecture. This result can also be deduced directly from \cite[Theorem~5.1]{RS11} in a different way.
\begin{theorem}
Let $f$ be a {\tef}. Either $I(f)$ is connected or it has infinitely many unbounded components.
\end{theorem}

{\it Acknowledgement}\quad
The authors are grateful to Walter Bergweiler and Dierk Schleicher for a discussion that led to the formulation of Theorem~1.3.

\section{Background material}
\setcounter{equation}{0}
We require several fundamental results from complex analysis, all of which can be found in~\cite{Pomm}, which we state here for the reader's convenience. The first two results concern the boundary behaviour of a conformal map~$f$ of the unit disc $\D=\{z:|z|<1\}$ into $\C$. Here {\capa} denotes logarithmic capacity and $\Lambda$ denotes linear measure, both of which are defined for Borel sets.
\begin{theorem}\label{Pomm0}
Suppose that $f:\D\to\C$ is a conformal map. Then for all $\zeta\in\partial \D$ apart from a set of capacity~0 the angular limit $f(\zeta)$ exists and is finite.
\end{theorem}
Theorem~\ref{Pomm0} is a classical result of Beurling \cite[Theorem~9.19]{Pomm}. Throughout the paper we use the notation $f(\zeta)$, where $\zeta\in \partial \D$, for the angular limit at~$\zeta$ of the conformal map~$f$, whenever this exists.

The second result on conformal maps \cite[Theorem~9.24]{Pomm} is a quantitative version of the fact that those boundary points of $f(\D)$ that can only be reached along relatively long paths in $f(\D)$ form a small subset of $\partial f(\D)$ in some sense.
\begin{theorem}\label{Pomm1}
Suppose that $f:\D\to\C$ is a conformal map, $V\subset f(\D)$ is open, $E\subset \partial\D$ is a Borel set, and $\alpha\in (0,1]$. If
\begin{itemize}
\item
${\rm dist}(f(0), V)\ge \alpha |f'(0)|,$
\item
$\Lambda(f(C)\cap V)\ge \beta>0$, for all curves $C$ in $\D$ that connect $0$ to $E$,
\end{itemize}
then
\[
\Lambda(E)\le 2\pi\, {\capa}E<\frac{15}{\sqrt{\alpha}}\exp\left(-\frac{\pi\beta^2}{{\are} V}\right).
\]
%Then the angular limit $F(\zeta)$ exists for all $\zeta\in\partial \D$ apart from a set of capacity~0.
%
%Also, if $R\ge 1$, then there is a set $E\subset \partial \D$ such that
%\[
%|F(\zeta)-F(0)|\le |F'(0)|R,\quad\text{for }\zeta\in \partial \D\setminus E,
%\]
%and
%\[
%\frac{\Lambda(E)}{2\pi}\le {\capa} E \le \frac{1}{\sqrt R}.
%\]
\end{theorem}

%Using a similar method to that in the proof of \cite[Corollary~9.20]{Pomm}, we deduce the following local version of Theorem~\ref{Pomm1} stated in the upper half-plane $\Ha=\{z:\Im z>0\}$.
%
%\begin{corollary}\label{Pomm1-cor}
%Let $\Phi:\Ha\to \C$ be a conformal map, and suppose that $a>0, b>0$ and $R\ge 1$. Then there is a set $E\subset [-a,a]$ such that
%\[
%|\Phi(t)-\Phi(ib)|\le 4R\,{\rm dist}(\Phi(ib),\partial\Phi(\Ha)),\quad\text{for }t\in [-a,a]\setminus E,
%\]
%and
%\[
%\frac{\Lambda(E)}{2\pi}\le {\capa} E \le \left(\frac{a^2+b^2}{2b}\right)\frac{1}{\sqrt R}.
%\]
%\end{corollary}
We also need various basic results on logarithmic capacity, which can be found in \cite[pages~204, 208 and 209]{Pomm}.
\begin{theorem}\label{Pomm2} Let $E$ and $E_n$, $n\ge 1$, be Borel subsets of $\C$.
\begin{itemize}
\item[(a)] If $E_1\subset E_2$, then ${\capa}E_1 \le {\capa}E_2$.
\item[(b)] If $\phi(z)=az+b$, then ${\capa}\phi(E)= |a|\,{\capa}E$.
\item[(c)] If $\phi$ is a Lipschitz map with constant~$M>0$, then ${\capa}\phi(E)\le M\,{\capa}E$.
\item[(d)] If $E=\bigcup_{n=1}^{\infty}E_n$ and ${\rm diam\,}E\le d$, then
\[
\frac{1}{\log(d/{\capa}E)}\le \sum_{n=1}^{\infty}\frac{1}{\log(d/{\capa}E_n)}.
\]
\item[(e)] The union of countably many sets of capacity zero has capacity zero.
\item[(f)] If~$\phi$ is a M\"obius transformation and~$E$ has capacity zero, then $\phi(E)$ has capacity zero.
\end{itemize}
\end{theorem}

Finally, we need Bagemihl's ambiguous point theorem; see \cite[Corollary~2.20]{Pomm}. Let~$f$ be a complex-valued function defined in~$\D$. For $\zeta\in \partial \D$ and a path $\gamma\subset \D$, we define the {\it cluster set} of~$f$ at~$\zeta$ along~$\gamma$ as follows:
\[
C_{\gamma}(f,\zeta)=\{w: \lim_{n\to\infty} f(z_n)= w,\;\text{for some sequence }z_n\in \gamma\;\text{such that } \lim_{n\to\infty}z_n=\zeta\}.
\]
A point $\zeta\in\partial \D$ is said to be an {\it ambiguous point} of~$f$ if there exist two paths~$\gamma$ and~$\gamma'$ in~$\D$ each ending at~$\zeta$ such that $C_{\gamma}(f,\zeta)\cap C_{\gamma'}(f,\zeta)=\emptyset$.

Bagemihl's theorem is as follows --- note that there are almost no hypotheses here about the function~$f$; it need not be continuous even.
\begin{theorem}\label{Bag}
Let~$f$ be a complex-valued function with domain~$\D$. Then~$f$ has at most countably many ambiguous points.
\end{theorem}
In fact we shall use the obvious adaptation of Theorem~\ref{Bag} from $\D$ to the upper half-plane $\Ha=\{z:\Im z>0\}$.

\section{A result on conformal maps}
\setcounter{equation}{0}
To prove Theorem~\ref{thm1.1} we require two results on the boundary behaviour of a conformal map, each of which states, roughly speaking, that if the map behaves in a certain way near a boundary point, then its boundary values behave in a similar way nearby. In \cite[page~220, Exercise~2]{Pomm} a result of this type is derived from Theorem~\ref{Pomm1}, but here we give more precise results of this type, again using Theorem~\ref{Pomm1}. For simplicity we state these results in the upper half-plane.
\begin{theorem}\label{conf-thm}
Let $\phi: \Ha \to \C$ be a conformal map, let $w_0\in\C\setminus \phi(\Ha)$, and let $\lambda>1$ and $\eps>0$. Also, for $n\ge 0$, put
\[
I_n=[\lambda^{n-1/2},\lambda^{n+1/2}]\quad\text{and}\quad E_n=\{t\in I_n: |\phi(t)-w_0|\ge \eps\},
\]
\[
J_n=[n-\tfrac12,n+\tfrac12] \quad\text{and}\quad F_n=\{t\in J_n: |\phi(t)-w_0|\ge \eps\}.
\]
\begin{itemize}
\item[(a)]
If $\phi(\lambda^n i)\to w_0$ as $n\to\infty$, then
\begin{equation}\label{sum1}
\sum_{n=0}^{\infty}\frac{1}{\log(\lambda^n/{\rm cap\,}E_n)}<\infty.
\end{equation}
\item[(b)]
If $\phi(n+i)\to w_0$ as $n\to \infty$, then
\begin{equation}\label{sum2}
\sum_{n=0}^{\infty}\frac{1}{\log(1/{\rm cap\,}F_n)}<\infty.
\end{equation}
\end{itemize}
\end{theorem}
\begin{proof}
We first prove part~(a). For $n\ge 0$, define
\begin{equation}\label{zn-Sn}
z_n=\lambda^n i\quad\text{and}\quad S_n=\{z\in\Ha:\lambda^{n-1}<|z|<\lambda^{n+1}\},
\end{equation}
so
\[
I_n=[\lambda^{n-1/2},\lambda^{n+1/2}] \subset\partial S_n \cap \partial \Ha.
\]
Then, for $n\ge 0$, let $\psi_n$ denote a conformal map of~$\D$ onto the semi-annulus $S_n$ such that $\psi_n(0)=z_n$. We can choose these maps so that $\psi_n=\lambda^n \psi_0$ for $n\ge 0$. Then, by Carath\'eodory's theorem \cite[page~24]{Pomm}, each $\psi_n$ extends to a continuous one-one map between the boundaries of $\partial \D$ and $\partial S_n$. By the reflection principle, this extension of $\psi_n$ is analytic with finite non-zero derivative on $\partial \D$, except at the four preimages of the vertices of $\partial S_n$.

%Let $\Phi(w)=1/(\phi(w)-d)$ for $w\in \Ha$, where $d\in J(f)$, so $\Phi$ is univalent in $\Ha$. Since $U$ is a {\bd}, we have
%\begin{equation}\label{Zn}
%z_n=\phi(w_n)=f^n(z_0)\to\infty\;\text{ as }n\to\infty,\quad\text{so}\quad \Phi(w_n)\to 0\;\text{ as }n\to\infty.
%\end{equation}

Now let $A(\eps)=\{w:\tfrac12\eps<|w-w_0|<\eps\}$ and, for $n\ge 0$, define
\begin{itemize}
\item
$V_n=\phi(S_n)\cap A(\eps)$,
\item
$E'_n=\{t\in E_n: t \text{ is not an ambiguous point of }\phi$\}.
%\setminus A_n$, where $A_n\subset \partial\D$ is the set of ambiguous points of the real function $|\Psi_n|$,
\end{itemize}
Then let the integer~$N$ be chosen so large that
\begin{equation}\label{Nlarge}
|\phi(z_n)-w_0|=|\phi(\lambda^n i)-w_0|\le \tfrac{1}{10} \eps,\quad\text{for }n\ge N.
\end{equation}
Note that if $n\ge N$ and $E_n\ne \emptyset$, then $V_n$ is non-empty.

We shall apply Theorem~\ref{Pomm1} to the conformal map $\Psi_n=\phi\circ \psi_n$, the non-empty open subset~$V_n$ of $\Psi_n(\D)$, and the Borel subset $\psi_n^{-1}(E'_n)$ of $\partial \D$, where $n\ge N$. We first show that
\begin{equation}\label{connect}
\Lambda(\Psi_n(C)\cap V_n)\ge \tfrac12 \eps, \;\text{for all curves } C \text{ in } \D \text{ that connect } 0 \text{ to } \psi_n^{-1}(E'_n).
\end{equation}
Indeed,
\[
|\Psi_n(0)-w_0|=|\phi(z_n)-w_0|\le \tfrac{1}{10} \eps,
 \]
by \eqref{Nlarge}, and if a curve~$C$ in $\D$ connects~$0$ to a point $\zeta\in \psi_n^{-1}(E'_n)$, then since $\Psi_n$ has an angular limit at~$\zeta$ such that $|\Psi_n(\zeta)-w_0|\ge \eps$ and $\zeta$ is not an ambiguous point of $\Psi_n$, we have
\[
\limsup_{s\to\zeta,\, s\in C} |\Psi_n(s)-w_0|\ge \eps.
\]
Hence $\Psi_n(C)$ must cross the annulus $A(\eps)$, passing through $V_n$, so \eqref{connect} holds.

Next, for $n\ge N$, by the definition of $V_n$ and \eqref{Nlarge},
\[
{\rm dist}(\Psi_n(0),V_n)={\rm dist}(\phi(z_n),V_n)\ge \tfrac12 \eps-\tfrac{1}{10}\eps = \tfrac{2}{5}\eps.
\]
Also, for $n\ge N$, we deduce by Koebe's theorem \cite[Corollary~1.4]{Pomm} and \eqref{Nlarge} that
\begin{align*}
|\Psi'_n(0)|&\le 4{\rm \,dist}(\Psi_n(0),\partial \Psi_n(\D))\\
&=4{\rm \,dist}(\phi(z_n),\partial \phi(S_n))\\
&\le 4|\phi(z_n)-w_0|\quad (\text{since }w_0\notin \phi(S_n))\\
&\le \tfrac25 \eps.
\end{align*}
Therefore, for $n\ge N$ we can indeed apply Theorem~\ref{Pomm1} to $\Psi_n$, $V_n$ and $\psi_n^{-1}(E'_n)$, with $\alpha=1$ and $\beta=\tfrac12\eps$ to give
\begin{equation}\label{capEn}
2\pi{\,\capa} \psi_n^{-1}(E'_n)\le 15\exp\left(-\frac{\tfrac14\pi\eps^2}{{\are} V_n}\right).
\end{equation}
As noted earlier, each $\psi_n$ extends analytically to most points of $\partial \D$ and in particular to the interior of the arc $\alpha_n=\partial \D \cap \psi_n^{-1}([\lambda^{n-1},\lambda^{n+1}])$ with non-zero derivative there. Also, $\psi_n=\lambda^n \psi_0$, for $n\ge 0$. Thus, by Theorem~\ref{Pomm2}, part~(c), and \eqref{capEn}, there exists an absolute constant $C>0$ such that, for $n\ge N$,
\begin{align}\label{cn}
{\capa} E'_n &\le \max\{|\psi'_n(\zeta)|:\zeta\in \alpha_n\}\,{\capa} \psi_n^{-1}(E'_n)\notag\\
&\le C\lambda^n \exp\left(-\frac{\tfrac14 \pi\eps^2}{{\are}\, V_n}\right).
\end{align}
Now $E_n$ differs from $E'_n$ by at most a countable set, by Theorem~\ref{Bag}, so
\begin{equation}\label{C-est}
\frac{1}{\log(C\lambda^n/{\capa} E_n)}\le \frac{{\are}\, V_n}{\tfrac14\pi\eps^2}.
\end{equation}
Since $V_n$, $n\ge 0$, are disjoint subsets of $A(\eps)$ we have $\sum_{n\ge 0} \,{\are} V_n \le{\are} A(\eps)$, so \eqref{sum1} follows immediately.
%Therefore, by Theorem~\ref{Pomm2}(d) and \eqref{cn}, we deduce that if $m\ge N$ and $K_m=\bigcup_{n\ge m}F'_n$, then
%\begin{align*}
%\frac{1}{\log(\sqrt{\lambda}/{\capa}K_m)}&\le \sum_{n\ge m}\frac{1}{\log(\sqrt{\lambda}/{\capa}F'_n)}\\
%&\le \sum_{n\ge m} \frac{1}{\pi(b-a)^2/{\rm area} V_n-\log (C/\sqrt{\lambda})}\\
%&=\sum_{n\ge m} \frac{{\rm area} V_n}{\pi(b-a)^2-(\log (C/\sqrt{\lambda}))\,{\are} V_n}\\
%& \to 0 \;\text{ as }m\to \infty,
%\end{align*}
%since $K_m\subset I_0$ for $m\ge N$ and $I_0$ has length $\sqrt{\lambda}-1/\sqrt{\lambda}<\sqrt{\lambda}$,

The proof of part~(b) proceeds in a similar way, on replacing \eqref{zn-Sn} by
\[
z_n=n+i\quad\text{and}\quad S_n=\{z:|\Re z-n|<1, 0<\Im z<2\}, \quad n\ge 0,
\]
and using the fact that $\phi(z_n)\to w_0$ as $n\to\infty$.
\end{proof}
{\it Remark}\; Note that it follows from \eqref{C-est} that the conclusion \eqref{sum1} can be slightly strengthened to
\[
\sum_{n=0}^{\infty} \frac{1}{\log(C\lambda^n/{\capa} E_n)}\le 3,
\]
where $C$ is a certain positive absolute constant, and similarly \eqref{sum2} can be strengthened to
\[
\sum_{n=0}^{\infty} \frac{1}{\log(C/{\capa} F_n)}\le 3.
\]

\section{Proof of Theorem~\ref{thm1.1}}
\setcounter{equation}{0} %
Without loss of generality we can assume that~$U$ is an invariant univalent {\bd} of~$f$. Then $U$ is {\sconn}, so there is a Riemann map $\phi$ from the upper half-plane~$\Ha$ to~$U$. Thus
\[
g=\phi^{-1}\circ f\circ \phi
\]
is a conformal map of $\Ha$ onto $\Ha$ and hence a M\"obius map with no fixed points in~$\Ha$. Therefore (see \cite[page~4]{aB91}, for example) we can choose~$\phi$ in such a way that $g$ is one of the following types:
\[
g(w)=\left\{
 \begin{array}{ll}
 \lambda w,& \text{where } \lambda>1,\\
 w+1.
 \end{array}
 \right.
\]
In the first case the Baker domain is said to be {\it hyperbolic} and in the second case it is said to be {\it parabolic}; see \cite{BF01}. In either case we have $g^n(w)\to\infty$ as $n\to \infty$ whenever $w\in \Ha$.

The idea of the proof is to show that for all $t\in \partial\Ha$, apart from a set of capacity zero, we have
%(which corresponds under a Moebius transformation from $\Ha$ to $\D$ to a set of capacity zero on $\partial \D$),
\begin{equation}\label{cond1}
\phi \text{ has a finite angular limit at } g^n(t) \text{ for all } n\ge 0,
\end{equation}
and
\begin{equation}\label{cond2}
\phi(g^n(t))\to\infty \text{ as } n\to \infty.
\end{equation}
The fact that \eqref{cond1} holds follows immediately from Theorem~\ref{Pomm0} and Theorem~\ref{Pomm2}, part~(f). Once \eqref{cond2} has been established we can deduce that for all $t\in\partial\Ha$, apart from a set of capacity zero, we have
\begin{align*}
\phi(g^n(t))&=\lim_{g^n(w)\to g^n(t)} \phi(g^n(w))\\
&=\lim_{w\to t} f^n(\phi(w))\\
&=f^n(\phi(t)),
\end{align*}
where the two limits are angular limits at boundary points of~$\Ha$, and hence $\phi(t)\in I(f)$. By Theorem~\ref{Pomm2}, part~(f), this is sufficient to prove the second statement of Theorem~\ref{thm1.1}.

Note that if the boundary of~$U$ is a Jordan curve through~$\infty$ (see \cite{BF01} for examples of this), then both \eqref{cond1} and \eqref{cond2} hold for all $t\in \partial \Ha$ with at most one exception, by Carath\'eodory's theorem, so all points of $\partial U$ are escaping with at most one exception.

We consider first the case when $g(w)=\lambda w$. Let $w_0=i$ and $z_0=\phi(w_0)\in U$. Then, for $n\ge 0$, define
\[w_n=g^n(w_0)=\lambda^n i.
\]

In order to be able to apply Theorem~\ref{conf-thm}, we let $\Phi(w)=1/(\phi(w)-c)$, where $c\notin U$. Then $\Phi$ is a conformal map of $\Ha$ into $\C\setminus\{0\}$. Since $U$ is a {\bd},
\begin{equation}\label{Zn}
\phi(w_n)=f^n(z_0)\to\infty\;\text{ as }n\to\infty,\quad\text{so}\quad \Phi(w_n)\to 0\;\text{ as }n\to\infty.
\end{equation}

Thus we can apply Theorem~\ref{conf-thm}, part~(a), with $w_0=0$,
\[
I_n=[\lambda^{n-1/2},\lambda^{n+1/2}] \quad \text{and}\quad E_n=\{t\in I_n: |\Phi(t)|\ge \eps\},\quad\text{for } n\ge 0,
\]
where $\eps>0$ is arbitrary, to deduce that
\begin{equation}\label{sum}
\sum_{n=0}^{\infty}\frac{1}{\log(\lambda^n/{\rm cap\,}E_n)}<\infty.
\end{equation}
For $n\ge 0$, let $\tilde E_n=E_n/\lambda^n \subset I_0$, so ${\capa}\tilde E_n={\capa}E_n/\lambda^n$, by Theorem~\ref{Pomm2}, part~(b). We deduce that if
\[K_m=\bigcup_{n\ge m}\tilde E_n,\;\text{ for } m\ge N,\]
then, by Theorem~\ref{Pomm2}, part~(d), and \eqref{sum}, together with the fact that the interval $I_0$ has length less than $\sqrt{\lambda}$,
\begin{align*}
\frac{1}{\log(\sqrt{\lambda}/{\capa}K_m)}&\le \sum_{n\ge m}\frac{1}{\log(\sqrt{\lambda}/{\capa}\tilde E_n)}\\
&=\sum_{n\ge m}\frac{1}{\log(\lambda^{n+1/2}/{\capa}E_n)}<\infty.
\end{align*}
It follows that ${\capa} K_m \to 0$ as $m\to \infty$, so the set
\begin{equation}\label{cap0}
K=\bigcap_{m\ge N}K_m=\bigcap_{m\ge N}\bigcup_{n\ge m} \tilde E_n\;\text{ has capacity } 0.
\end{equation}
Now~$K$ consists of those points $t\in I_0$ such that for infinitely many~$n$ the angular limit of $\Phi$ at $g^n(t)=\lambda^n t$ exists and lies in $\{z:|z|\ge \eps\}$. Since the set of points where $\Phi$ has no angular limit is of capacity~$0$, by Theorem~\ref{Pomm0}, it follows that, for all~$t$ in $I_0$ apart from a set of capacity zero, we have
\[
|\Phi(g^n(t))|< \eps,  \text{ for all sufficiently large } n.
\]
Since $\eps>0$ is arbitrary and a countable union of sets of capacity zero has capacity zero, we deduce that, for all~$t$ in $I_0$ apart from a set of capacity zero, we have
\[
\Phi(g^n(t))\to 0 \; \text{ as } n\to \infty
\]
and hence
\[
\phi(g^n(t))\to \infty \; \text{ as } n\to \infty.
\]
Since $g(w)=\lambda w$, it follows readily that this property holds for all $t\in\partial \Ha$ apart from an exceptional set of capacity zero, which proves \eqref{cond2}.

The proof of \eqref{cond2} in the case $g(w)=w+1$ proceeds in a similar manner with $w_0=i$, $z_0=\phi(w_0)\in U$ and
\[w_n=g^n(w_0)=n+i,\quad\text{so}\quad \phi(w_n)=f^n(z_0) \to\infty\;\text{ as } n\to\infty,
\]
by using Theorem~\ref{conf-thm}, part~(b). We omit the details.

\section{Proof of Theorem~\ref{thm1.2}}
\setcounter{equation}{0} %
Suppose that~$f$ is a {\tef}, that~$E$ is a subset of $I(f)$ such that $J(f)\subset\overline{E}$, and that~$E$ is contained in the union of finitely many components of $I(f)$. Under these hypotheses we proved in \cite[Theorem~5.1]{RS11} that
\begin{itemize}
 \item[(a)]\(I(f)\cap J(f)\) is contained in one component, $I_1$ say, of
$I(f)$;
 \item[(b)]all the components of $I(f)$ are unbounded, and they consist of
\begin{itemize}
 \item[(i)] $I_1$, which also contains any escaping {\wand}s and
any Baker domains of $f$ with at least one boundary point in
$I(f)$,
 \item[(ii)] any Baker domains of $f$ with no boundary points in
$I(f)$ and the infinitely many preimage components of such
Baker domains.
\end{itemize}
\end{itemize}
We shall show that under these hypotheses Baker domains whose boundaries do not meet $I(f)$ cannot occur, so the result from \cite{RS11} quoted above implies that $I(f)$ has just one component. Therefore, if $I(f)$ is disconnected, then $E$ meets infinitely many components of $I(f)$, as required.

Suppose then that~$U$ is a Baker domain of~$f$ whose boundary does not meet $I(f)$. By Corollary~\ref{cor1.1}, the boundary of~$U$ is disconnected. Then $U$ has more than one complementary component, each of which is closed and unbounded, and meets $J(f)$.

We now note that $J(f)=\overline{I(f)\cap J(f)}$, which follows by the blowing up property of $J(f)$ and the fact that $I(f)\cap J(f)\ne\emptyset$; see~\cite{aE89}. All the points of $I(f)\cap J(f)$ lie in complementary components of $U$, and $I(f)\cap J(f)$ cannot be contained in a single complementary component of~$U$ because $J(f)=\overline{I(f)\cap J(f)}$. Hence the component~$I_1$ of $I(f)$ in part~(a) meets at least two complementary components of $U$ and so it must meet the boundaries of these two complementary components, which are subsets of $\partial U$. This contradicts the fact that $\partial U\cap I(f)=\emptyset$. Hence such a Baker domain cannot exist in this case. This completes the proof of Theorem~\ref{thm1.2}.

\end{document}